\documentclass{amsart}
\usepackage{xcolor}
\usepackage{amsmath,amsfonts,amssymb,color,hyperref, enumerate,tabularx}
\newcommand*{\backin}{\rotatebox[origin=c]{-180}{$\in$}}
\usepackage{amsmath, amsfonts,amsthm,amssymb,amscd, verbatim,graphicx,color,multirow,tikz,tikz-cd,booktabs, caption, mathdots,bm,chngcntr,adjustbox,longtable
}
\usepackage{tikz-cd}
\usetikzlibrary{positioning}
\newtheorem{theorem}{Theorem}[section]
\newtheorem{lemma}[theorem]{Lemma}
\newtheorem{proposition}[theorem]{Proposition}
\newtheorem{notation}[theorem]{Notation}
\newtheorem{definition}[theorem]{Definition}

\begin{document}

\title[Normal $2$-coverings]{Normal $2$-coverings in affine groups of small dimension}

\author[M. Fusari]{Marco Fusari}
\address{Marco Fusari, Dipartimento di Matematica ``Felice Casorati", University of Pavia, Via Ferrata 5, 27100 Pavia, Italy}
\email{lucamarcofusari@gmail.com}

\author[A.~Previtali]{Andrea Previtali}
\address{Andrea Previtali, Dipartimento di Matematica e Applicazioni, University of Milano-Bicocca, Via Cozzi 55, 20125 Milano, Italy}
\email{andrea.previtali@unimib.it}

\author[P.~Spiga]{Pablo Spiga}
\address{Pablo Spiga, Dipartimento di Matematica e Applicazioni, University of Milano-Bicocca, Via Cozzi 55, 20125 Milano, Italy}
\email{pablo.spiga@unimib.it}

\begin{abstract}
A finite group $G$ admits a normal $2$-covering if there exist two proper subgroups $H$ and $K$ with $G=\bigcup_{g\in G}H^g\cup\bigcup_{g\in G}K^g$. For determining inductively the finite groups admitting a normal $2$-covering, it is important to determine all finite groups $G$ possessing a normal $2$-covering, where no proper quotient of $G$ admits such a covering. 

Using terminology arising from the O'Nan-Scott theorem, Garonzi and Lucchini have shown that these groups fall into four natural classes: product action, almost simple, affine and diagonal.

In this paper, we start a preliminary investigation of the affine case.  

\begin{center}
\textit{With deep appreciation to Martino Garonzi and Andrea Lucchini, for keeping us entertained.}
\end{center}
\end{abstract}


\maketitle
\section{Introduction}\label{intro}
In finite group theory, Jordan's theorem is a very beautiful and useful result: a finite group cannot be covered by the conjugates of a proper subgroup. That is, if $G$ is a group, $H$ is a subgroup of $G$, and $G = \bigcup_{g \in G} H^g$, then  $G = H$. This result is frequently applied in the context of permutation groups, where it can be rephrased as saying that every transitive group of degree greater than $1$ admits a derangement (that is, a permutation with no fixed points).

With this in mind, it is natural to look into the groups $G$ that admit two proper subgroups whose conjugates together cover the whole of $G$. More generally, a \textit{\textbf{normal $k$-covering}} of $G$ is a collection of $k$ proper subgroups $H_1, \ldots, H_k$ of $G$ such that every element of $G$ is conjugate to an element of some $H_i$. That is,
$$G = \bigcup_{i=1}^k \bigcup_{g \in G} H_i^g.$$
The smallest such $k \in \mathbb{N}$ for which $G$ admits a normal $k$-covering is called the \textit{\textbf{normal covering number}} of $G$.

It is not particularly meaningful to ask for a complete classification of the finite groups $G$ with $\gamma(G) = 2$. In fact, if $G$ is any group with a normal subgroup $N$ such that $\gamma(G/N) = 2$, then $\gamma(G) = 2$ as well. This follows because the pullback, under the natural projection, of  two proper subgroups of $G/N$ witnessing a normal $2$-covering yields a normal $2$-covering for $G$. From this perspective, if we want to truly understand which finite groups admit a normal $2$-covering, it makes more sense to focus on the case where:
\begin{itemize}
\item $\gamma(G) = 2$, and
\item $\gamma(G/N) \ne 2$, for every non-identity normal subgroup $N$ of $G$.
\end{itemize}
In this paper, a group with this property is called \textit{\textbf{basic}}. Thus, basic
groups are the building objects having normal covering number 2, because any other such group is a group extension of one of these building blocks.

The perspective outlined above is precisely the one taken by Garonzi and Lucchini in~\cite{GL}. To provide some context for our work, we briefly summarize the key elements of their contribution. Suppose $G$ is a basic group, and let $H$ and $K$ be proper subgroups of $G$ that realize one of these normal $2$-coverings. By suitably replacing $H$ and $K$ if needed, we may assume that both are maximal subgroups of $G$.

A fundamental result in~\cite{GL} is that such a group $G$ must have a unique minimal normal subgroup, denoted by $N$. Furthermore, they prove that either $G$ is almost simple, or—after possibly swapping $H$ and $K$—we have $N \le H$. The case where $G$ is almost simple and neither $H$ nor $K$ contains the socle of $G$ has been fully analyzed in the monograph~\cite{libro}. Without delving into technical details, the basic groups in this class are found to be highly constrained: apart from a few small examples arising as ``low-level noise'', the remaining groups belong to natural families whose coverings admit geometric or combinatorial interpretations.
 Thus, from this point on, we restrict our attention to the scenario where $N \le H$.

Given that $\gamma(G/N) > 2$, it follows that $N \nleq K$, so $K$ must be a maximal subgroup of $G$ with trivial core. This means the natural action of $G$ on the set of right cosets of $K$ yields a faithful and primitive permutation representation of $G$, where $K$ acts as a point stabilizer. In this context, Garonzi and Lucchini employ the language of the O'Nan–Scott classification of primitive groups, showing that $G$ falls into one of the following types: almost simple, affine, product action, or diagonal. 

The classification of basic diagonal groups was completed in~\cite{FPS}, where it is shown (again) that such groups are extremely constrained in structure and can be described in detail.

The classification of basic almost simple and product action groups remains open, and its resolution is closely connected to the theory of Shintani descent. Some progress has been reported by Scott Harper~\cite{scott}, building on techniques developed in his monograph~\cite{scott1}.

In this paper, we initiate a study of the class of \emph{basic affine groups}. For completeness, we restate the formal definition below.
\begin{definition}\label{der}{\rm
Let $G$ be a finite group. We say that $G$ is a \textit{\textbf{basic group of affine type}} if the following conditions hold:
\begin{itemize}
    \item $G$ has a unique minimal normal subgroup $V$;
    \item $V$ is abelian;
    \item $\gamma(G/V) \ne 2$;
    \item $\gamma(G) = 2$.
\end{itemize}
Since our focus in this paper is on this class of groups, we adopt the following terminology: a finite group $G$ is said to be \textbf{\textit{non-basic (of affine type)}} if it satisfies the first three conditions above, but $
\gamma(G) \ne 2.$}
\end{definition}

Before presenting our main results, we offer some preliminary remarks and informal observations. Computational data, obtained using the software \textsc{Magma}~\cite{magma}, suggests that the vast majority of primitive subgroups of the affine semilinear group $\mathrm{A}\Gamma\mathrm{L}_d(q)$ appear to be basic for small values of $d$. This observation is based solely on computational evidence and stands in sharp contrast to the classifications presented in~\cite{libro,FPS}, where the basic groups (of almost simple and diagonal type) do show a very constrained structure. As a result, our theorems are formulated to identify and describe the non-basic affine groups.

Our first main theorem is an auxiliary result and provides a complete classification of supersolvable groups with normal covering number~$2$.
\begin{theorem}\label{thrm:1}
Let $G$ be a supersolvable group. Then $\gamma(G)\ne 2$ if and only if $G$ is nilpotent.
\end{theorem}
In our view, this result is of independent interest and also plays a key role in our analysis of basic affine groups of small dimension. 

Throughout the remainder of the paper, let $p$ be a prime number, and let $f$ and $d$ be positive integers. Set $q = p^f$. We denote by $\mathrm{A}\Gamma\mathrm{L}_d(q)$ the affine semilinear group, and let $V$ denote its socle. Thus
\[
\mathrm{A}\Gamma\mathrm{L}_d(q) = V \rtimes \Gamma\mathrm{L}_d(q),
\]
where $\Gamma\mathrm{L}_d(q) = \mathrm{GL}_d(q) \rtimes \mathrm{Aut}(\mathbb{F}_q)$, and $\mathbb{F}_q$ denotes the finite field of order~$q$.

\begin{theorem}\label{thrm:2}
 Let $G=V\rtimes H$ be a primitive subgroup of $\mathrm{A}\Gamma\mathrm{L}_1(q)$, where $H\le\Gamma\mathrm{L}_1(q)$ is the stabilizer of the zero vector. Then $G$ is non-basic if and only if 
\begin{itemize}
\item $p\equiv 3\pmod 4$,
\item $f=2$,
\item $H$ is a dihedral $2$-group and, using Notation~$\ref{not}$, $H$ is $\mathrm{GL}_1(q)$-conjugate to  $\langle\sigma,\lambda^{(p-1)e}\rangle$ for some divisor $e$ of $p+1$.
\end{itemize}
\end{theorem}
To deal with $2$-dimensional affine groups we actually need a slightly stronger (but more technical) statement, see Theorem~\ref{thrm:mainmain}.

\begin{theorem}\label{thrm:3}
Let $G=V\rtimes H$ be a primitive subgroup of $\mathrm{A}\Gamma\mathrm{L}_2(q)$, where  $H\le\Gamma\mathrm{L}_2(q)$ is the stabilizer of the zero vector. If $G$ is non-basic, then $f$ and $|H|$ are powers of $2$.
\end{theorem}

In our view, the pattern exhibited by $\mathrm{A}\Gamma\mathrm{L}_d(q)$ for $d \in \{1,2\}$ is quite striking: non-basic  groups appear to arise from irreducible representations of certain $2$-groups. However, this pattern already breaks down when $d = 3$. In Section~\ref{sec:example}, we show that if $f$ is even, $p \ne 3$, and $q \notin \{4,25\}$, then the group $G = V \rtimes \mathrm{SU}_3(q^{1/2})$ is non-basic. Motivated by this, we propose the following conjecture.

\begin{theorem}\label{conj}
Let $G = V \rtimes H$ be a primitive subgroup of $\mathrm{A}\Gamma\mathrm{L}_3(q)$, where $H \le \Gamma\mathrm{L}_3(q)$ is the stabilizer of the zero vector. If $G$ is non-basic, then either $H$ is a $\{2,3\}$-group, or $f$ is even and $\mathrm{SU}_3(q^{1/2}) \unlhd H$.
\end{theorem}
Computational evidence suggests that Theorem~\ref{conj} can be refined further, by proving that, in the first alternative, $H$ is a $3$-group.

We conclude this introductory section with a fundamental result due to Garonzi and Lucchini~\cite{GL}. Let $G = V \rtimes H$ be an affine primitive group. Define $H^*$ to be the set of elements in $H$ that fix some nonzero vector in $V$. That is,
\[
H^* = \{ h \in H \mid v h = v \text{ for some } v \in V \setminus \{0\} \} = \{ h \in H \mid \mathbf{C}_V(h) \ne 0 \}.
\]

\begin{proposition}[{\cite[Corollary~14]{GL}}]\label{prop:gl}
Let $G$ be as above. Then $G$ is basic (i.e., $\gamma(G) = 2$) if and only if there exists a proper subgroup $T < H$ such that
\[
H^* \subseteq \bigcup_{h \in H} T^h.
\]
\end{proposition}

\section{Proof of Theorem~\ref{thrm:1}}\label{sec:easy}
Assume first that $G$ is nilpotent. If $\gamma(G) = 2$, then $G$ admits two maximal subgroups $H$ and $K$ such that
\[
G = \bigcup_{g \in G} H^g \cup \bigcup_{g \in G} K^g.
\]
As $G$ is nilpotent, $H, K \unlhd G$, and hence $G = H \cup K$; however, this is impossible. Thus, $\gamma(G) \ne 2$.

Conversely, suppose that $G$ is supersolvable and $\gamma(G) \ne 2$. We argue by contradiction and suppose that $G$ is not nilpotent. From~\cite[5.2.4]{robinson}, there exists a maximal subgroup $M$ of $G$ such that $M$ is not normal in $G$. Let 
\[
K = \bigcap_{g \in G} M^g
\]
be the core of $M$ in $G$, and let $\bar{G} = G/K$. As $\gamma(G) \ne 2$, we have $\gamma(\bar{G}) \ne 2$. Since $\bar{G}$ is a quotient of $G$, it is also supersolvable. Moreover, from~\cite[5.2.4]{robinson}, as $M/K$ is a maximal subgroup of $\bar{G}$ that is not normal, we deduce that $\bar{G}$ is not nilpotent. Therefore, replacing $G$ with $\bar{G}$ if necessary, we may assume that $G = \bar{G}$, that is, $K = 1$.

Let $N$ be a minimal normal subgroup of $G$. Since $M$ is maximal and core-free in $G$, we have $G = NM$ and $N \cap M=1$. Moreover, as $G$ is supersolvable, $N$ is cyclic of prime order. Let $C = {\bf C}_G(N)$. Now, $C \cap M \unlhd M$, and $C \cap M$ is also normalized by $N$. Thus, $C \cap M \unlhd NM=G$. As $M$ is core-free in $G$, we have $C\cap M=1$ and hence $C=N$ from Dedekind modular law. This shows that $M$ acts faithfully by conjugation on $N$. As $N$ is cyclic of prime order, $G = N \rtimes M$ is a Frobenius group with kernel $N$ and complement $M$. As Frobenius groups have normal covering number $2$, we deduce that $\gamma(G) = 2$, which is a contradiction. This concludes the proof of Theorem~\ref{thrm:1}.

\smallskip

 Proposition~\ref{prop:gl} admits a significant simplification when $H$ is nilpotent.
\begin{lemma}\label{lemma:mercoledi0}
Let $G=V\rtimes H$ be an affine primitive group with $H$ nilpotent. Then $G$ is non-basic (that is, $\gamma(G)\ne 2$) if and only if $H=\langle H^\ast\rangle$.
\end{lemma}
\begin{proof}
Since $G$ is primitive, $V$ is the unique minimal normal subgroup of $G$. Moreover, since $H$ is nilpotent, by Theorem~\ref{thrm:1}, we have $\gamma(G/V)=\gamma(H)\ne 2$. Therefore, from Proposition~\ref{prop:gl}, $G$ is basic if and only if $H^\ast$ is contained in $\bigcup_{h \in H} T^h$ for some subgroup $T < H$. Replacing $T$ with a maximal subgroup of $H$ if necessary, we may suppose that $T\unlhd H$. Therefore,  $G$ is basic if and only if $H^\ast$ is contained in $T$ for some subgroup $T < H$.
\end{proof}

\section{Proof of Theorem~\ref{thrm:2}}\label{sec:thrm2}

In this section, we prove a slightly stronger version of Theorem~\ref{thrm:2}. We begin by introducing some notation that will be used throughout the section.

\begin{notation}\label{not}{\rm
Let $p$ be a prime number, let $f$ be a positive integer, let $q = p^f$, let $\mathbb{F}_q$ be the finite field of cardinality $q$, let $\lambda$ be a generator of the multiplicative group of $\mathbb{F}_q$, and let $\sigma$ be the Frobenius automorphism of $\mathbb{F}_q$. Thus $x^\sigma = x^p$ for all $x \in \mathbb{F}_p$.

With this notation, we have:
\begin{align*}
\Gamma\mathrm{L}_1(q) &= \mathbb{F}_q^\ast \rtimes \mathrm{Aut}(\mathbb{F}_q) = \langle \lambda \rangle \rtimes \langle \sigma \rangle,\\
\mathrm{A}\Gamma\mathrm{L}_1(q) &= \mathbb{F}_q \rtimes \Gamma\mathrm{L}_1(q) = \mathbb{F}_q \rtimes \langle \lambda, \sigma \rangle.
\end{align*}
Note that $\Gamma\mathrm{L}_1(q)$ is metacyclic, and hence so is each of its subgroups.

Let $H$ be a subgroup of $\Gamma\mathrm{L}_1(q)$. In particular, we may write $H = \langle \lambda^d, \sigma^j \lambda^i \rangle$, where $d$ divides $q - 1$, $j$ divides $f$, and $i \in \mathbb{Z}$. We may assume that $\langle \lambda^d \rangle = H \cap \langle \lambda \rangle$. Since $|H : \langle \lambda^d \rangle| = |\langle \sigma^j \rangle|$, it follows that
$$(\sigma^j\lambda^i)^{f/j} = \lambda^{i\frac{p^f - 1}{p^j - 1}} \in H \cap \langle \lambda \rangle = \langle \lambda^d \rangle,$$
that is,
\begin{equation}\label{eq:neq1}
d\, \text{ divides }\, i\frac{p^f - 1}{p^j - 1}.
\end{equation}

An arbitrary element $h$ of $H$ can be written (uniquely) as $h=(\sigma^j\lambda^i)^k\lambda^{d\ell}$, for a unique $k\in \{0,\ldots,f/j-1\}$ and a unique $\ell\in \{0,\ldots,(p^f-1)/d-1\}$. We have
\begin{equation}\label{often}
h=(\sigma^j\lambda^i)^k\lambda^{d\ell}=\sigma^{ijk}\lambda^{i\frac{p^{jk}-1}{p^j-1}+d\ell}.
\end{equation}
We use very often this computation in this section.
}
\end{notation}

\begin{lemma}\label{lemma:mercoledi1}
Let $H = \langle \lambda^d, \sigma^j \lambda^i \rangle$ be as in Notation~$\ref{not}$ and let $h = (\sigma^j \lambda^i)^k \lambda^{d\ell} \in H$. Then $h \in H^\ast$ if and only if $p^{\gcd(f,jk)} - 1$ divides $i(p^{jk} - 1)/(p^j - 1) + d\ell$.
\end{lemma}

\begin{proof}
By definition, $h \in H^\ast$ if and only if there exists $t \in \mathbb{Z}$ such that $\lambda^t h = \lambda^t$. From~\eqref{often},
we have
$$\lambda^t h = \lambda^{t p^{jk} + i\frac{p^{jk} - 1}{p^j - 1} + d\ell}.$$
Therefore, $\lambda^t h = \lambda^t$ if and only if
$$t(p^{jk} - 1) \equiv -i\frac{p^{jk} - 1}{p^j - 1} - d\ell \pmod{p^f - 1}.$$
This congruence has a solution in $t$ if and only if $\gcd(p^{jk} - 1, p^f - 1)$ divides $-i\frac{p^{jk} - 1}{p^j - 1} - d\ell$.

Finally, by an inductive argument (see also~\cite[Lemma~2.2~(1)]{libro}), we have
$$\gcd(p^{jk} - 1, p^f - 1) = p^{\gcd(jk,f)} - 1$$
as claimed.
\end{proof}

\begin{lemma}\label{lemma:mercoledi2}
Let $H = \langle \sigma^j \lambda^i, \lambda^d \rangle$ be as in Notation~$\ref{not}$. If $H = \langle H^\ast \rangle$, then $p^j - 1$ divides both $d$ and $i$.

In particular, replacing $H$ by a suitable $\langle \lambda \rangle$-conjugate, we may assume $H = \langle \sigma^j,\lambda^d \rangle$.
\end{lemma}
\begin{proof}
Set $t=p^j$, $\tau=\sigma^j$ and $\mu=\lambda^i$. By hypothesis $\langle H^*\rangle=H$, in particular $\lambda^d$ is the
  product of elements $(\tau\mu)^k\lambda^{d\ell}$ satisfying the clause in Lemma~\ref{lemma:mercoledi1}.
  We prove by induction on $r$ where $\lambda^d=\prod_{i=1}^{r}(\tau\mu)^{k_i}\lambda^{d\ell_i}$ that
  the RHS\footnote{the Right Hand Side.} has shape $\tau^K\lambda^a$, where $K=\sum_{i=1}^{r}k_i$ and $t-1$ divides $a$.
 By Equation \ref{often} $(\tau\mu)^k\lambda^{d\ell}=\tau^k\lambda^a$, where $a=t^kd\ell+it\frac{t^k-1}{t-1}$ and we are done when $r=1$.
 Now 
 $$\tau^x\lambda^y\tau^z\lambda^w=\tau^{x+z}\lambda^{yt^z+w}$$
 shows that the exponents of $\tau$ add up and $y,w\equiv 0 \mod t-1$ force $yt^z+w\equiv 0\mod t-1$. It follows that $t-1$ divides $d$.
 Analogously $\tau\mu=\tau^K\lambda^a$, thus $t-1\mid a=i$. Now conjugation via $\lambda^\ell$ where $(t-1)\ell-i=0$ fixes $\lambda^d$ and maps $\sigma^j\lambda^i$ to $\sigma^j$.
\end{proof}

Given a positive integer $n$, we let $\pi(n)$ denote the set of prime divisors of $n$.

\begin{lemma}\label{lemma:mercoledi3}
Let $H=\langle \sigma^j\lambda^i,\lambda^d\rangle$ be as in Notation~$\ref{not}$. Then $H$ nilpotent if and only if $\pi((p^f-1)/d)\subseteq \pi(p^j-1)$.
\end{lemma}
\begin{proof}
We denote by $\gamma_\kappa(H)$ the $\kappa$th term in the lower central series of $H$. As 
\begin{align*}
[\sigma^j\lambda^i,\lambda^d]&=
\lambda^{-i}\sigma^{-j}\lambda^{-d}\sigma^j\lambda^i\lambda^d=
\lambda^{-i}\lambda^{-dp^j}\lambda^{i+d}=
\lambda^{d(1-p^j)},
\end{align*} we deduce that $\gamma_2(H)$ is generated by $\lambda^{d(1-p^j)}$. Therefore, arguing inductively, we obtain that $\gamma_{\kappa}(H)$ is generated by $\lambda^{d(1-p^j)^{\kappa-1}}$. Therefore, $H$ is nilpotent if and only if $\pi(p^f-1)\subseteq \pi(d(p^j-1))$, that is, $\pi((p^f-1)/d)\subseteq\pi(p^j-1)$.\end{proof}

\begin{lemma}\label{lemma:mercoledi33}
Let $H=\langle \sigma^j\lambda^i,\lambda^d\rangle$ be as in Notation~$\ref{not}$ with $H$ nilpotent. Then, each prime in $\pi((p^f-1)/d)\setminus\pi(f/j)$ divides $|H:\langle H^\ast\rangle|$.

In particular, if $H=\langle H^\ast\rangle$, then $\pi((p^f-1)/d)\subseteq \pi(f/j)$.
\end{lemma}
Given a group $G$ and a group element $g$, we let ${\bf o}(g)$ denote the order of $g$. 
\begin{proof}
Let $r$ be a prime number. Since $H$ is nilpotent, we have $H=R\times Q$, where $R$ is a Sylow $r$-subgroup of $H$ and $Q$ is a Hall $r'$-subgroup of $H$. Let $h\in H^\ast$ with $h=g_rg_{r'}$, where $g_r\in R$ and $g_{r'}\in Q$. Let $\alpha,\beta\in\mathbb{Z}$ with $\alpha{\bf o}(g_r)+\beta{\bf o}(g_{r'})=1$. We have
\begin{align*}
H^\ast\backin& h^{\beta{\bf o}(g_{r'})}=g_r^{\beta{\bf o}(g_{r'})}g_{r'}^{\beta{\bf o}(g_{r'})}=g_r^{1-\alpha{\bf o}(g_r)}=g_r,\\
H^\ast\backin& h^{\alpha{\bf o}(g_{r})}=g_r^{\alpha{\bf o}(g_{r})}g_{r'}^{\alpha{\bf o}(g_{r})}=g_{r'}^{1-\beta{\bf o}(g_{r'})}=g_{r'}.
\end{align*}
Thus $g_r\in H^\ast\cap R=R^\ast$, $g_{r'}\in H^\ast\cap Q=Q^\ast$ and $g_rg_{r'}=h\in R^\ast\times Q^\ast$. Since $h$ is an arbitrary element of $H^\ast$, we have shown $H^\ast\subseteq R^\ast\times Q^\ast$.

Let $r\in\pi((p^f-1)/d)\setminus\pi(f/j)$. Thus $|R|$ is relatively prime to $|H:\langle \lambda^d\rangle|$ and hence $R\le\langle\lambda^d\rangle$. Since the non-identity elements in $\langle\lambda^d\rangle$ fix no non-zero vector, we deduce $R^\ast=1$ and $\langle R^\ast\rangle=1$. Thus $|H:\langle H^\ast\rangle|=|H:\langle Q^\ast\rangle|$ is divisible by $r$.
\end{proof}

Given a positive integer $n$ and a prime number $r$, we let $v_r(n)$ be the largest integer with $r^{v_r(n)}$ dividing $n$. We now need two numerical results.

\begin{lemma}\label{lemma:mercoledi5}
Let $r$ be a prime number, let $n$ be a positive integer and let $x\in\mathbb{Z}$ with $v_r(x)\ge 1$. Then
\begin{align*}
(1+x)^n\equiv
\begin{cases}
1+nx(x+2)/2\pmod {r^{v_r(nx(x+2))}}&\textrm{when }r=2, v_r(x)=1,\\
&n \textrm{ is twice an odd number},\\
1+nx\pmod{r^{v_r(nx)+1}}&\textrm{otherwise}.
\end{cases}
\end{align*}
\end{lemma}
\begin{proof}
We argue by induction on $n$. Assume first that $n$ is a prime number. From the binomial theorem, we have
\begin{align*}
(1+x)^n &=
\sum_{k=0}^n\binom{n}{k}x^k = 1 + nx + \sum_{k=2}^n\binom{n}{k}x^k.
\end{align*}
If $r\ne n$, then $v_r(nx) = v_r(x)$. Thus, for $k\ge 2$, we have
\[
v_r\left(\binom{n}{k}x^k\right) \ge v_r(x^k) = kv_r(x) \ge 2v_r(x) \ge v_r(x)+1,
\]
where the last inequality follows from the fact that $v_r(x)\ge 1$. Thus $(1+x)^n \equiv 1+nx \pmod{r^{v_r(nx)+1}}$. If $r=n$, then $v_r(n)=1$ and, for $2 \le k \le n-1$, we have
\[
v_r\left(\binom{n}{k}x^k\right) = 1 + v_r(x^k) = 1 + kv_r(x) \ge 1 + 2v_r(x) \ge v_r(nx)+1.
\]
When $k=n$, we have
\[
v_r\left(\binom{n}{k}x^n\right) = v_r(x^n) = nv_r(x) \ge v_r(nx)+1,
\]
where the last inequality holds when $r$ is odd, or when $r=2$ and $v_r(x)\ge 2$.

The general case now follows by induction on $n$.
\end{proof}

Given a positive integer $n$, we denote with $n_{2'}$ the largest odd divisor of $n$.  In particular, $n_{2'}=n/2^{v_2(n)}$.

\begin{lemma}\label{lemma:mercoledi4}
Let $H$ be as in Notation~$\ref{not}$ and assume $H=\langle H^\ast\rangle$ and $H$  nilpotent. 

If $p=2$, or if $p^j\equiv 1\pmod 4$, or if $f/j$ is odd, then $(p^f-1)/d$ divides $f/j$.
Moreover, when $p^j\equiv 3\pmod 4$ and $f/j$ is even, then $(p^f-1)/d$ divides
$(f/j)2^{v_2(p^j+1)-v_2(d)}$.

In particular, in all cases $((p^f-1)/d)_{2'}$ divides $f/j$.
\end{lemma}
\begin{proof}
  Let $r$ be a prime divisor of $(p^f-1)/d$ and, when $r=2$, suppose $4$ divides $p^j-1$. Since $H$ is nilpotent, from Lemma~\ref{lemma:mercoledi3}, we have $r\in \pi((p^f-1)/d)\subseteq \pi(p^j-1)$ and hence $v_r(p^j-1)>0$. Thus we may write $p^j=1+r^{v_r(p^j-1)}u$ with $u$ coprime to $r$.
  Then, by Lemma~\ref{lemma:mercoledi5}, we have
\begin{align*}
p^f&=(p^j)^{f/j}=(1+r^{v_r(p^j-1)}u)^{f/j}\\
&\equiv 
1+(f/j)r^{v_r(p^j-1)}u\pmod{r^{v_r(f/j)+v_r(p^j-1)+1}}
\end{align*}
  and hence $v_r(p^f-1)=v_r(p^j-1)+v_r(f/j)$.

From Lemma~\ref{lemma:mercoledi2}, $p^j-1$ divides $d$ and hence $v_r(p^j-1)\le v_r(d)$. Therefore, we have
\begin{align*}
v_r((p^f-1)/d)&=v_r(p^f-1)-v_r(d)=v_r(p^j-1)+v_r(f/j)-v_r(d)\\
&=v_r(f/j)-(v_r(d)-v_r(p^j-1))\le v_r(f/j).
\end{align*}
This implies that the largest power of $r$ dividing $(p^f-1)/d$ divides $f/j$. Since this holds for every $r$, the proof now follows.

When $p\ne 2$, $p^j\equiv 3\pmod 4$ and $f/j$ is even, the argument above needs to be adjusted when $r=2$. Indeed, from  Lemma~\ref{lemma:mercoledi5}, we have
\begin{align*}
p^f&=(p^j)^{f/j}=(1+(p^j-1))^{f/j}\equiv 1+\frac{f(p^{2j}-1)}{2j}\pmod{r^{v_r(f/j)+v_r(p^{2j}-1)}}
\end{align*}
and hence $v_2(p^f-1)=v_2(f/j)+v_2(p^{2j}-1)-1=v_2(f/j)+v_2(p^j+1)$. The rest of the argument now follows.
\end{proof}

Before presenting the main result of this section, we first require an auxiliary result from representation theory. We are grateful to Gabriel Navarro for highlighting the importance of reference~\cite{Navarro} in relation to our work.

\begin{lemma}\label{lemma:navarro}
Let $J$ be a normal subgroup of a finite group $H$ with $H/J$ abelian. Let $V$ be an irreducible $\mathbb{F}_pH$-module and let $W$ be an irreducible $\mathbb{F}_pJ$-submodule of $V$. Then $\dim V/\dim W$ is an integer dividing $|H:J|$.
\end{lemma}
\begin{proof}
The analogous statement for irreducible modules over the complex numbers is classical, follows from Clifford's theorem, and can be found in~\cite[page~84 and Chapter~11]{Isaacs}. In this case, it is not required that $H/J$ is abelian. The analogous result over arbitrary fields is given in~\cite{Navarro}.
\end{proof}

Let $n$ be a positive integer, and let $b$ be an integer relatively prime to $n$. We denote by ${\bf o}_n(b)$ the least positive integer $\ell$ such that $n$ divides $b^\ell - 1$. In other words, ${\bf o}_n(b)$ is the order of the residue class $b \bmod n$ in the group of units of the ring $\mathbb{Z}/n\mathbb{Z}$. By Fermat's little theorem, ${\bf o}_n(b)$ divides $\varphi(n)$, where $\varphi$ is Euler's totient function.

The following is the main result of this section, from which Theorem~\ref{thrm:2} will follow immediately.
\begin{theorem}\label{thrm:mainmain}
Let $G=V\rtimes H$ be a primitive subgroup of $\mathrm{A}\Gamma\mathrm{L}_1(q)$
 with $H$ nilpotent and let $r$ be the largest prime dividing $|H|$. If $r>2$, then  $r$ divides $|H:\langle H^\ast\rangle|$; in particular, $\gamma(G)=2$ and $G$ is basic.
\end{theorem}
\begin{proof}
Let $K=\langle H^\ast\rangle$.
As $H^\ast\subseteq K$, we have $K^\ast=H^\ast$ and hence $K=\langle K^\ast\rangle$. Using Notation~\ref{not}, we have
$H=\langle\sigma^j\lambda^i,\lambda^d\rangle$. Moreover, applying Notation~\ref{not} and Lemma~\ref{lemma:mercoledi2} to $K$, replacing $K$ and $H$ with a suitable $\langle\lambda\rangle$-conjugate, we may suppose $K=\langle \sigma^{j_0},\lambda^{d_0}\rangle$.

We argue by contradiction and we suppose $|H:K|$ is not divisible by $r$, where $r>2$ is the largest prime divisor of $|H|$. As $|H|=(f/j)(p^f-1)/d$ and $|K|=(f/j_0)(p^f-1)/d_0$, we deduce that
\begin{equation}\label{eq:today2}
|H:K|=\frac{j_0}{j}\cdot \frac{d_0}{d}\,\hbox{ is relatively prime to }r.
\end{equation}

Assume  $r\in \pi(f/j)\setminus\pi((p^f-1)/d)$. Let $R$ be a Sylow $r$-subgroup of $H$ and let $Q$ be a Hall $r'$-subgroup of $H$. As $H$ is nilpotent, $H=R\times Q$. Since $|H:K|$ is relatively prime to $r$, we have $R\le K$ and, as $\gcd(r,(p^f-1)/d_0)=1$, we have $R\le \langle \sigma^{j_0}\rangle$ and $R$ is abelian. In particular,
$$H\le {\bf C}_{\Gamma\mathrm{L}_1(p^f)}(R)\le {\bf C}_{\Gamma\mathrm{L}_1(p^f)}(\sigma^{f/r})=\mathbb{F}_{p^{f/r}}^\ast\rtimes \langle \sigma\rangle.$$
Observe now that the subfield $\mathbb{F}_{p^{f/r}}$ is $H$-invariant, because it is invariant by $\mathbb{F}_{p^{f/r}}^\ast\rtimes \langle \sigma\rangle$. However, this contradicts the fact that $H$ acts irreducibly on $\mathbb{F}_{p^f}$. This contradiction has shown that $r\in\pi((p^f-1)/d)$ and hence, from~\eqref{eq:today2}, we get
\begin{equation}\label{eq:allprimes}
r\in \pi((p^f-1)/d_0)\cap \pi(f/j_0).
\end{equation}

 From~\eqref{eq:allprimes}, $r\in \pi((p^f-1)/d_0)$ and hence, from Lemma~\ref{lemma:mercoledi4} applied to $K$, we have
\begin{align}\label{eq:today3}
v_r(f/j_0)\ge v_r((p^f-1)/d_0).
\end{align}

Now consider  $J=\langle \sigma^{j_0},\lambda^d\rangle$ and  $W=\mathbb{F}_p[\lambda^d]$, the subfield of $\mathbb{F}_{p^f}$ containing $\lambda^d$. Clearly, $K\le J\le H$ and $J\unlhd H$. Observe that $W$ is an $\mathbb{F}_pJ$-module because $W$ is invariant under $\lambda^d$ and $\sigma^{j_0}$. Moreover, $W$ is irreducible as $\mathbb{F}_pJ$-module from its definition. Since $W$ is a subfield of $\mathbb{F}_{p^f}$, we have  $W=\mathbb{F}_{p^{f_0}}$ for some divisor $f_0$ of $f$. Observe that $f_0$ is the smallest positive integer with ${\bf o}(\lambda^d)=(p^f-1)/d$ dividing $p^{f_0}-1$. In other words, $f_0={\bf o}_{(p^f-1)/d}(p)$ and, by raising $p$ to the $j_0$th power, we get $f_0/\gcd(f_0,j_0)={\bf o}_{(p^f-1)/d}(p^{j_0})$. From Euler's theorem,
\begin{equation}\label{eq:today}\frac{f_0}{\gcd(f_0,j_0)} \hbox{ divides }\varphi((p^f-1)/d).
\end{equation}

Observe now that $J\unlhd H$ and $|H:J|$ is relatively prime to  $r$. Therefore, from Lemma~\ref{lemma:navarro}, $f/f_0$ is relatively prime to $r$. Thus, from~\eqref{eq:today2} and~\eqref{eq:today}, 
\begin{equation}\label{eq:today1}
v_r(f/j)\le v_r(\varphi((p^f-1)/d)).
\end{equation}
From~\eqref{eq:today2},~\eqref{eq:today3} and~\eqref{eq:today1}, we have
\begin{align}\label{eq:ultimate}
v_r(f/j_0)&\ge v_r((p^f-1)/d_0)=v_r((p^f-1)/d)>v_r(\varphi((p^f-1)/d))\ge v_r(f/j),
\end{align}
where the strict inequality 
$v_r((p^f-1)/d)>v_r(\varphi((p^f-1)/d))$
follows from the fact that $r$ is the largest odd prime dividing $|H|$. However, the chain of inequalities in~\eqref{eq:ultimate} gives a contradiction.

Finally, as $H\ne \langle H^\ast\rangle$, from Lemma~\ref{lemma:mercoledi0}, $\gamma(G)=2$ and $G$ is basic.
\end{proof}

Using Theorem~\ref{thrm:mainmain}, we are now ready to prove Theorem~\ref{thrm:2}.

\begin{proof}[Proof of Theorem~$\ref{thrm:2}$]
Let $G=V\rtimes H$ be a primitive subgroup of $\mathrm{A}\Gamma\mathrm{L}_1(q)$, where $H\le\Gamma\mathrm{L}_1(q)$ is the stabilizer of the zero vector. We use Notation~\ref{not} for $H$.

Assume first that $p\equiv 3\pmod 4$, $f=2$, $H$ is a dihedral $2$-group and  $H$ is $\mathrm{GL}_1(p^f)$-conjugate to $H=\langle\sigma,\lambda^{(p-1)e}\rangle$ for some divisor $e$ of $p+1$. From Lemma~\ref{lemma:mercoledi1}, $\sigma,\sigma \lambda^{(p-1)e}\in H^\ast$ and hence $\langle H^\ast\rangle\ge\langle \sigma,\sigma \lambda^{(p-1)e}\rangle=H$. Thus, by Lemma~\ref{lemma:mercoledi0}, $\gamma(G)\ne 2$ and $G$ is non-basic.

Conversely, suppose that $G$ is non-basic. As $\Gamma\mathrm{L}_1(q)$ is metacyclic, it is supersolvable and hence so is $H$, because it is a subgroup of a supersolvable group. From Theorem~\ref{thrm:1}, $H$ is nilpotent, because $\gamma(H)=\gamma(G/V)\ne 2$. Thus, from Lemma~\ref{lemma:mercoledi0}, we have $H=\langle H^\ast\rangle$. From Theorem~\ref{thrm:mainmain}, $H$ is a $2$-group. Moreover, from Lemma~\ref{lemma:mercoledi2} applied to $H$, replacing $H$ with a suitable $\langle\lambda\rangle$-conjugate, we have $H=\langle \sigma^{j},\lambda^{d}\rangle$ with $d$ divisible by $p^j-1$.

As $H$ is a $2$-group, we have $(p^f-1)/d=2^a$ and $f/j=2^b$, for some $a,b\in\mathbb{N}$. In particular, $p$ is an odd prime. Let $W=\mathbb{F}_p[\lambda^d]$. Arguing as in the proof of Theorem~\ref{thrm:mainmain}, $W$ is an irreducible $\mathbb{F}_pH$-submodule of $\mathbb{F}_{p^f}$. The irreducibility of $H$ implies $W=\mathbb{F}_{p^f}$. Therefore, $f$ is the smallest positive integer with $(p^f-1)/d=2^a$ dividing $p^f-1$. Thus $f={\bf o}_{(p^f-1)/d}(p)$. Therefore, from Euler's theorem, $f=2^bj$ divides $\varphi((p^f-1)/d)=\varphi(2^a)=2^{a-1}$, that is, 
\begin{equation}\label{eqevening}
j=1 \hbox{ and } b\le a-1.
\end{equation}
If $p^j\equiv 1\pmod 4$ or if $2^b=1$, then from Lemma~\ref{lemma:mercoledi4} $2^a=(p^f-1)/d$ divides $f=2^b$. Thus $a\le b$, which contradicts~\eqref{eqevening}. Thus $p^j\equiv 3\pmod 4$ and $b\ge 1$. As $p^j\equiv 3\pmod 4$, we have $p\equiv 3\pmod 4$.

Assume $f>2$. From Lemma~\ref{lemma:mercoledi5}, $$p^{f/2}=(1+(p^{2}-1))^{f/4}\equiv 1+\frac{f}{4}(p^2-1)\pmod {2^{v_2(f(p^2-1)/4)+1}}.$$
Thus $v_2(p^{f/2}-1)=v_{2}(f(p^{2}-1)/4)=v_2(f)+v_2(p+1)-1$. An entirely analogous argument gives $v_2(p^f-1)=v_2(p^{f/2}-1)+1$. Now, $(p^f-1)/d=2^a$ and hence, as $p-1$ divides $d$, we have $$a=v_2((p^f-1)/d)=v_2(p^f-1)-v_2(d)\le v_2(p^f-1)-1=v_2(p^{f/2}-1).$$ Thus $2^a$ divides $p^{f/2}-1$ and we contradict $f={\bf o}_{2^a}(p)$. Therefore $f=2$.

Finally, write $d=(p-1)e$, for some divisor $e$ of $p+1$. We have
\begin{align*}
(\lambda^{d})^\sigma&=(\lambda^{(p-1)e})^\sigma=\lambda^{(p-1)ep}=\lambda^{p^2e}\lambda^{-pe}=\lambda^e\lambda^{-pe}
\\
&=\lambda^{-(p-1)e}=(\lambda^d)^{-1}.
\end{align*}
Therefore $H$ is a dihedral group.
\end{proof}

\section{Proof of Theorem~\ref{thrm:3}}
Using Theorem~\ref{thrm:mainmain}, the proof of Theorem~\ref{thrm:3} is quite natural.

We adopt a notation similar to the notation in Section~\ref{sec:thrm2}. Let $p$ be a prime number, let $f$ be a positive integer, let $q=p^f$ and let $\sigma:\mathbb{F}_q\to\mathbb{F}_q$ be the Frobenius automorphism. Let $\mathrm{A}\Gamma\mathrm{L}_2(q)$ be the affine semilinear group and 
let $V=\mathbb{F}_q^2$ be its socle, consisting of the translations. We also let $Z$ be the center of $\mathrm{GL}_2(q)$. We let $G=V\rtimes H$ where $H\le\Gamma\mathrm{L}_2(q)$ and assume that
\begin{itemize}
\item $H$ acts irreducibly on $V$ and $\gamma(H)>2$,
\item $\gamma(G)\ne 2$, that is, $G$ is non-basic.
\end{itemize}

\begin{proof}[Proof of Theorem~$\ref{thrm:3}$]

We study the various possibilities depending on the Aschbacher class of $H$.

\smallskip

\noindent\textsc{Case $H$ is contained in a maximal subgroup in the Aschbacher class $\mathcal{C}_1$} 

\smallskip

\noindent In this case, $H$ stabilizes a $1$-dimensional $\mathbb{F}_q$-subspace  of $V$. However, this contradicts the fact that $H$ acts irreducibly on $V$.

\smallskip

\noindent\textsc{Case $H$ is contained in a maximal subgroup in the Aschbacher class $\mathcal{C}_2$} 

\smallskip

\noindent In this case, $H$ is contained in the group $(D\rtimes\langle\iota\rangle)\rtimes\mathrm{Aut}(\mathbb{F}_{p^f})=D\rtimes\langle\iota,\sigma\rangle$, where
$$D=
\left\{
\begin{pmatrix}
a&0\\
0&b
\end{pmatrix}
\mid a,b\in \mathbb{F}_q\setminus\{0\}
\right\} \hbox{ and }
\iota=
\begin{pmatrix}
0&1\\
1&0
\end{pmatrix}.$$
Observe that $$Z,\,D,\,D\rtimes\langle\iota\rangle$$
are normal subgroups of $D\rtimes\langle\iota,\sigma\rangle$ and that $Z\cong\mathbb{F}_q^\ast$, $D/Z\cong\mathbb{F}_q^\ast$, $D\rtimes\langle\iota\rangle/D\cong\langle\iota\rangle$
 and $(D\rtimes\langle\iota\rangle)\rtimes\langle\sigma\rangle/(D\rtimes\langle\iota\rangle)\cong\langle\sigma\rangle$. Therefore, $(D\rtimes\langle\iota\rangle)\rtimes\langle\sigma\rangle$ is supersolvable. In particular, $H$ is supersolvable and hence, by Theorem~\ref{thrm:1}, $H$ is nilpotent, because $\gamma(H)\ne 2$. Thus, by Lemma~\ref{lemma:mercoledi0}, $H=\langle H^\ast\rangle$.

Since $H$ is contained in $D\rtimes\langle\iota,\sigma\rangle$, $H$ preserves the direct sum decomposition $\mathbb{F}_qe_1\oplus\mathbb{F}_qe_2$. Moreover, since $H$ acts irreducibly on $V$, $H$ contains an element $h$ swapping the two direct summands $\mathbb{F}_qe_1\oplus\mathbb{F}_qe_2$. In particular,
$$h=\sigma^j
\begin{pmatrix}
0&a\\
b&0\end{pmatrix},$$
for some divisor $j$ of $f$ and some $a,b\in\mathbb{F}_q^\ast$. Let $M$ be the index-2 subgroup of $H$ stabilizing the two direct summands $\mathbb{F}_qe_1$ and $\mathbb{F}_qe_2$ of $V$. As $|H:M|=2$, replacing $h$ by a suitable power, we may suppose that $h$ is a $2$-element.


For $i\in \{1,2\}$, let $\pi_i:M\to\Gamma\mathrm{L}_1(q)$ be the homomorphism induced by the action of $M$ on $\mathbb{F}_qe_i$. Let $K_i$ be the image of $\pi_i$. Observe that $K_i$ acts irreducibly on $\mathbb{F}_qe_i$; indeed, if $K_i$ leaves a proper $\mathbb{F}_p$-subspace $U_i$ of $\mathbb{F}_qe_i$ invariant, then $U_i\oplus U_i^h$ is a proper $\mathbb{F}_p$-subspace of $V$ invariant by $H$, which contradicts the fact that $V$ is an irreducible $\mathbb{F}_pH$-module. Let $G_i=\mathbb{F}_qe_i\rtimes K_i$.

If $\gamma(G_i)\ne 2$, then Theorem~\ref{thrm:1} implies that $K_i$ is a $2$-group and $q=p^2$ for some odd prime $p$. Thus $f=2$. As $M$ is a subdirect subgroup of $K_1\times K_2$, $M$ is a $2$-group. Since $|H:M|=2$, $H$ is also a $2$-group.

Assume $\gamma(G_i)= 2$ and that $|K_i|$ is divisible by some odd prime number. From Theorem~\ref{thrm:mainmain}, $|K_i:\langle K_i^\ast\rangle|$ is divisible by some odd prime number $r$. 
 Let $M_i$ be the preimage of $\langle K_i^\ast\rangle$ by $\pi_i$. As $h$ normalizes $M$ and swaps the two direct summands $\mathbb{F}_qe_1$ and $\mathbb{F}_qe_2$ of $V$, we have $M_1^h=M_2$ and $M_2^h=M_1$.
 
Let $R$ be a Sylow $r$-subgroup of $H$ and let $Q$ be a Hall $r'$-subgroup of $H$. As $H=R\times Q$, we get $M_i=(R\cap M_i)\times (Q\cap M_i)$. Since $h$ is an element of $2$-power order and since $h$ swaps $M_1$ and $M_2$, we deduce
$$R\cap M_1=R\cap M_2.$$ 
Observe also that, as $r$ divides $|M:M_i|$, $R\cap M_1=R\cap M_2$ is a proper subgroup of $M$.

Since $H=R\times Q$ and since $R$ and $Q$ have relatively prime orders, we have $H^\ast\subseteq R^\ast\times Q^\ast$, see also the proof of Lemma~\ref{lemma:mercoledi33}. As $H=\langle H^\ast\rangle$, we deduce that $R^\ast$ generates $R$. However this is impossible. Indeed, since $r$ is odd, $R\le M$ and hence $R^\ast\subseteq M^\ast\subseteq M_1\cup M_2$. Thus $R^\ast\subseteq R\cap (M_1\cup M_2)=(R\cap M_1)\cup (R\cap M_2)=R\cap M_1$. This contradiction has shown that $H$ is a $2$-group.

The fact that $f$ is also a power of $2$ follows from the fact that $H$ acts irreducibly on $V$, indeed, if $f$ is divisible by an odd prime, then a suitable conjugate of $H$ fixes a proper subspace of $V$ where the scalars are in a proper subfield of $\mathbb{F}_q$.

\smallskip

\noindent\textsc{Case $H$ is contained in a maximal subgroup in the Aschbacher class $\mathcal{C}_3$} 

\smallskip

\noindent In this case, $H$ preserves an extension field and hence $G$ is a subgroup of the affine semilinear group $\mathrm{A}\Gamma\mathrm{L}_1(q^2)$. From Theorem~\ref{thrm:1}, $H$ is a dihedral $2$-group and $q=p\equiv 3\pmod 4$.

\smallskip

\noindent\textsc{Case $H$ is contained in a maximal subgroup in the Aschbacher class $\mathcal{C}_5$} 

\smallskip

\noindent In this case, $H$ preserves a subfield space. In particular, $q=q_0^{f_0}$, for some divisor $f_0>1$ of $f$, and $H\le\mathrm{GL}_2(q_0)\mathrm{Aut}(\mathbb{F}_q)$. Moreover, considering the cases we have excluded above, we may suppose that $\mathrm{SL}_2(q_0)\le H$. Set $A=\mathrm{GL}_2(q_0)\mathrm{Aut}(\mathbb{F}_q)$. We let $B={\bf N}_A(B_0)$, where $B_0$ is a Borel subgroup of $\mathrm{GL}_2(q_0)$. Also, we let $C={\bf N}_A(C_0)$, where $C_0$ is a Singer cycle of $\mathrm{GL}_2(q_0)$. We have
$$A=\bigcup_{a\in A}B^a\cup\bigcup_{a\in A}C^a.$$ 
Moreover, from $\mathrm{SL}_2(q_0)\le H$, we get $A=BH$ and $A=CH$ and hence by intersecting both sides of the previous union with $H$, we get
$$H=\bigcup_{a\in H}(B\cap H)^a\cup\bigcup_{a\in H}(C\cap H)^a.$$
However, this contradicts the fact that $\gamma(H)>2$.

\smallskip

\noindent\textsc{Case $H$ is contained in a maximal subgroup in the Aschbacher class $\mathcal{C}_6$} 

\smallskip

\noindent Let $M$ be the maximal subgroup in the class $\mathcal{C}_6$ with $H\cap\mathrm{SL}_2(q)\le M$. From~\cite[Table~8.1]{bray}, we see that one of the following holds
\begin{itemize}
\item $M\cong 2_-^{1+2}.S_3=2^\cdot S_4^-$, $q=p=\pm 1\pmod 8$, ${\bf N}_{\mathrm{GL}_2(q)}(M)=Z M$, 
\item $M\cong 2_-^{1+2}:3=2^\cdot A_4$, $q=p=\pm 3,5,\pm 13\pmod {40}$, ${\bf N}_{\mathrm{GL}_2(q)}(M)=Z M.2$.
\end{itemize}
We consider these two possibilities in turn.

In the first case, we have $\Gamma\mathrm{L}_2(q)=\mathrm{GL}_2(q)$ and $M\le H\le ZM$. Since $Z$ is central and $M$ has normal covering number $2$, we get $\gamma(H)=2$, contradicting $\gamma(H)>2$. In the second case, we have $\Gamma\mathrm{L}_2(q)=\mathrm{GL}_2(q).2$ and $M\le H$ and $|H:H\cap ZM|\le 2$. When $|H:H\cap ZM|=1$, then $H/(Z\cap M)\cong \mathrm{Alt}(4)$; whereas, when $|H:H\cap ZM|=2$, then $H/(Z\cap M)\cong \mathrm{Sym}(4)$. In both cases, $Z$ is central and $\mathrm{Alt}(4),\mathrm{Sym}(4)$ have normal covering number $2$, we get $\gamma(H)=2$, contradicting $\gamma(H)>2$.

\smallskip

\noindent\textsc{Case $H$ is contained in a maximal subgroup in the Aschbacher class $\mathcal{S}$} 

\smallskip

\noindent Let $M$ be the maximal subgroup in the class $\mathcal{S}$ with $H\cap\mathrm{SL}_2(q)\le M$. From~\cite[Table~8.2]{bray}, we see that one of the following holds
\begin{itemize}
\item $M\cong 2^\cdot A_5$, $q=p=\pm 1\pmod {10}$, ${\bf N}_{\mathrm{GL}_2(q)}(M)=Z M$, 
\item $M\cong 2^\cdot A_5$, $q=p^2$, $p=\pm 3\pmod {10}$, ${\bf N}_{\mathrm{GL}_2(q)}(M)=Z M.2$.
\end{itemize}
The argument here is exactly as in the previous case, in either case we contradict $\gamma(H)=2$.
\end{proof}

\section{An example in dimension 3}\label{sec:example}Let $p$ be a prime, let $f$ be even, let $q=p^f$ and let $q_0 = p^{f/2}=q^{1/2}$. 
Suppose $p \ne 3$ and $q_0 \notin \{2,5\}$. Consider the special unitary group
\[
H = \mathrm{SU}_3(q_0) \le \mathrm{GL}_3(q),
\]
defined via the Hermitian form associated with the matrix
\[
J =
\begin{pmatrix}
0 & 0 & 1 \\
0 & 1 & 0 \\
1 & 0 & 0
\end{pmatrix}
\]
with respect to the standard basis $e_1, e_2, e_3$ of the vector space $V = \mathbb{F}_q^3$. Thus, $H$ consists of all matrices $A \in \mathrm{SL}_3(q)$ such that $A J \bar A^{T} = J$, where $\bar A$ denotes the matrix obtained by raising each entry of $A$ to the $q_0$-th power, and $A^T$ is the transpose of $A$.

Now, consider the semidirect product $G = V \rtimes H$, equipped with its natural affine action on $V$. The group $G$ acts primitively on $V$ because $V$ is an irreducible $\mathbb{F}_p H$-module (see, for instance,~\cite[Proposition~2.10.6]{kl}). From~\cite[Lemma~4.2]{libro}, we know that $\gamma(H) > 2$, i.e., $H$ does not admit a normal $2$-covering.\footnote{We exclude the cases where $q_0 \in \{2,5\}$ or where $q_0$ is a power of $3$, to ensure that $\gamma(H) > 2$.}

We aim to show that $\gamma(G) > 2$. According to~\cite[Corollary~14]{GL} (see Proposition~\ref{prop:gl}), it suffices to prove that
\[
H^\ast = \{h \in H \mid \mathbf{C}_V(h) \ne 0\}
\]
is not contained in $\bigcup_{h \in H} T^h$ for any proper subgroup $T < H$. We proceed by contradiction, assuming that such a proper subgroup $T$ of $H$ exists. Without loss of generality, we may assume that $T$ is maximal in $H$.

It is not difficult to verify that $H$ has $q_0 + 1$ orbits on $V$ (see \cite[Lemma~2.10.5~(ii)]{kl}), one of which consists of the zero vector. Another orbit consists of all totally singular vectors. For instance, the vector $e_1$ is totally singular, and
\begin{equation}\label{eqq0}
\mathbf{C}_H(e_1) =
\left\{
\begin{pmatrix}
1 & 0 & 0 \\
a & 1 & 0 \\
d & -a^{-q_0} & 1
\end{pmatrix}
\;\middle|\;
a,d \in \mathbb{F}_q,\ d + d^{q_0} + a a^{q_0} = 0
\right\}.
\end{equation}
The remaining orbits of $H$ on $V$ consist of non-degenerate vectors and split into $q_0 - 1$ distinct $H$-orbits, classified by their norm. For example, $e_2$ is non-degenerate, and
\begin{equation}\label{eqq1}
\mathbf{C}_H(e_2) =
\left\{
\begin{pmatrix}
a & 0 & b \\
0 & 1 & 0 \\
c & 0 & d
\end{pmatrix}
\;\middle|\;
\begin{array}{l}
a,d \in \mathbb{F}_q,\ ad - bc = 1, \\
a^{q_0} c + a c^{q_0} = 0, \\
c^{q_0} d + c d^{q_0} = 0, \\
bc^{q_0} + b^{q_0} c = 1
\end{array}
\right\}\cong\mathrm{SU}_2(q)
\cong \mathrm{SL}_2(q_0),
\end{equation}
where the last isomorphism follows from~\cite[Proposition~2.9.1~(i)]{kl}.

From~\eqref{eqq0}, we deduce that $T$ contains unipotent elements with Jordan forms $J_3$ and $J_1 \oplus J_2$.\footnote{Here, $J_\ell$ denotes a unipotent element represented by a Jordan block of size $\ell \times \ell$.} 

Furthermore,~\eqref{eqq1} implies that $T$ contains a semisimple element $x$ of order $q_0 + 1$, as $\mathrm{SL}_2(q_0)$ does. The element $x$ fixes a $1$-dimensional subspace of $V$ and acts as a Singer cycle on its orthogonal complement. Similarly,~\eqref{eqq1} shows that $T$ contains a semisimple element $y$ of order $q_0 - 1$, again because $\mathrm{SL}_2(q_0)$ does. Since $\gcd(q_0 - 1, q_0 + 1) = \gcd(q - 1, 2)$, the order of $T$ is divisible by
\[
\frac{p(q_0 - 1)(q_0 + 1)}{\gcd(q - 1, 2)} = \frac{p(q - 1)}{\gcd(q - 1, 2)}.
\]

Taking into account that $T$ contains unipotent elements with Jordan forms $J_3$ and $J_1 \oplus J_2$ and that its order is divisible by $p(q - 1)/\gcd(q - 1, 2)$, a detailed case-by-case analysis of the maximal subgroups of $\mathrm{SU}_3(q_0)$, as listed in~\cite[Tables~8.5 and~8.6]{bray}, confirms that no such subgroup $T$ can exist.

\section{Proof of Theorem~\ref{conj}}
We only provide a sketch of the proof, as the argument follows verbatim the proof of Theorem~\ref{thrm:3}. Indeed, we may analyze the various possibilities for $H$ depending on the Aschbacher class of $H$.

When $H$ belongs to the Aschbacher classes $\mathcal{C}_1$, $\mathcal{C}_3$, $\mathcal{C}_5$, or $\mathcal{S}$, the argument is nearly identical and requires only minor cosmetic changes.

When $H$ lies in the Aschbacher class $\mathcal{C}_8$, then either $\mathrm{SO}_3(q) \unlhd H$ or $\mathrm{SU}_3(\sqrt{q}) \unlhd H$. In the second case, we obtain the examples described in Section~\ref{sec:example}. In the first case, since $\Omega_3(q) \cong \mathrm{PSL}_2(q)$, we deduce $\gamma(H)=2$ by arguing as in the proof of Theorem~\ref{thrm:3} (in the case where $H$ belongs to the Aschbacher class $\mathcal{C}_5$). However, this contradicts the hypothesis that $\gamma(H)\ne 2$.

Thus, it remains to consider the case where $H$ is in the Aschbacher class $\mathcal{C}_2$. We provide complete details for this case, because the argument is quite delicate. Let $V=\mathbb{F}_q^3$ and let $e_1,e_2,e_3$ be the canonical basis of $V$.  Here $H$ is contained in the group $(D\rtimes\langle\iota_2,\iota_3\rangle)\rtimes\mathrm{Aut}(\mathbb{F}_{p^f})=D\rtimes\langle\iota_2,\iota_3,\sigma\rangle$, where
$$D=
\left\{
\begin{pmatrix}
a&0&0\\
0&b&0\\
0&0&c
\end{pmatrix}
\mid a,b,c\in \mathbb{F}_q\setminus\{0\}
\right\},\,
\iota_2=
\begin{pmatrix}
0&1&0\\
1&0&0\\
0&0&1\\
\end{pmatrix}  \hbox{ and }
\iota_3=\begin{pmatrix}
0&1&0\\
0&0&1\\
1&0&0\\
\end{pmatrix}
.$$
If the action induced by $H$ by conjugation on $\{\mathbb{F}_qe_1,\mathbb{F}_qe_2,\mathbb{F}_qe_3\}$ is the symmetric group $\mathrm{Sym}(3)$, then $H$ has an epimorphic image isomorphic to $\mathrm{Sym}(3)$ and hence $\gamma(H)=2$, because $\mathrm{Sym}(3)$ has a normal $2$-covering. However, this contradicts the hypothesis $\gamma(H)\ne 2$. Therefore, the action induced by $H$ by conjugation on $\{\mathbb{F}_qe_1,\mathbb{F}_qe_2,\mathbb{F}_qe_3\}$ is the alternating group $\mathrm{Alt}(3)$ and hence $H$ is contained in $D\rtimes\langle\iota_3,\sigma\rangle$.

Let $Z$ be the subgroup of $\mathrm{GL}_3(q)$ consisting of the scalar matrices. Observe that $$Z,\,D$$
are normal subgroups of $D\rtimes\langle\iota_3,\sigma\rangle$ and that $Z\cong\mathbb{F}_q^\ast$, $D/Z\cong\mathbb{F}_q^\ast\times \mathbb{F}_q^\ast$ and $D\rtimes\langle\iota_3,\sigma\rangle/D\cong\langle\iota_3\rangle\times\langle\sigma\rangle$. In particular, every chief factor of $H$ is either cyclic of prime order, or the direct product of two cyclic groups of prime order.

We argue by contradiction and we suppose that $H$ is not supersolvable. From~\cite[5.4.7]{robinson}, $H$ admits a maximal subgroup $M$ with $|H:M|$ not prime. Let $K$ be the core of $M$ in $H$, let $\bar H=H/K$ and adopt the bar ``notation''. The permutation action of $H$ on the right cosets of $M$ endows $H$ of the structure of a primitive solvable group. In particular, $\bar H$ is a primitive soluble permutation group with point stabilizer $\bar M$. Let $\bar V$ be the socle of $\bar H$ (we also let $V$ be the preimage of $\bar V$ in $H$ under the natural projection onto $\bar H$). Then $\bar V$ is an elementary abelian normal subgroup of $\bar H$ and $\bar H=\bar V\rtimes \bar M$. Then $|\bar H:\bar M|=|H:M|=|\bar V|$ and, from the previous paragraph, $|\bar V|=r^2$, for some prime number $r$. In particular, $\bar H\le \mathrm{AGL}_2(r)$. Since $\gamma(H)\ne 2$, we have $\gamma(\bar H)\ne 2$ and hence, from Theorem~\ref{thrm:3} applied to $\bar H$, we deduce $|\bar M|$ is a power of $2$. In particular, $3$ is relatively prime to $|\bar M|=|\bar H/\bar V|=|H:V|$. 

We now claim that $H=K(H\cap \langle D,\sigma\rangle)$. Indeed, as $H\le \langle D,\sigma,\iota_3\rangle$ and $|\langle D,\sigma,\iota_3\rangle:\langle D,\sigma\rangle|=3$, if $H\ne K(H\cap \langle D,\sigma\rangle)$, then $K(H\cap \langle D,\sigma\rangle)$ has index $3$ in $H$. This implies that $\bar H$ has a normal subgroup having index $3$, contradicting the previous paragraph. Therefore $H=K(H\cap \langle D,\sigma\rangle)$, as claimed.

Now, $$\bar H=\frac{H}{K}=\frac{K(H\cap \langle D,\sigma\rangle)}{K}\cong \frac{H\cap \langle D,\sigma\rangle}{K\cap \langle D,\sigma\rangle}.$$
However, this is a contradiction because $\bar H$ is not supersolvable, whereas $\langle D,\sigma\rangle$ is supersolvable. This contradiction has finally shown that $H$ is supersolvable. Therefore, from Theorem~\ref{thrm:1}, $H$ is nilpotent and, by Lemma~\ref{lemma:mercoledi0}, $H=\langle H^\ast\rangle$.

Since $H$ is contained in $D\rtimes\langle\iota_3,\sigma\rangle$, $H$ preserves the direct sum decomposition $\mathbb{F}_qe_1\oplus\mathbb{F}_qe_2\oplus\mathbb{F}_qe_3$. Moreover, since $H$ acts irreducibly on $V$, $H$ contains an element $h$ acting cyclically on the three direct summands $\mathbb{F}_qe_1\oplus\mathbb{F}_qe_2\oplus\mathbb{F}_qe_3$. In particular,
$$h=\sigma^j
\begin{pmatrix}
0&a&0\\
0&0&b\\
c&0&0
\end{pmatrix},$$
for some divisor $j$ of $f$ and some $a,b,c\in\mathbb{F}_q^\ast$. Let $M$ be the index-3 subgroup of $H$ stabilizing the three direct summands $\mathbb{F}_qe_1$, $\mathbb{F}_qe_2$ and $\mathbb{F}_qe_3$ of $V$. As $|H:M|=3$, replacing $h$ by a suitable power, we may suppose that $h$ is a $3$-element.

For $i\in \{1,2,3\}$, let $\pi_i:M\to\Gamma\mathrm{L}_1(q)$ be the homomorphism induced by the action of $M$ on $\mathbb{F}_qe_i$. Let $K_i$ be the image of $\pi_i$. Observe that $K_i$ acts irreducibly on $\mathbb{F}_qe_i$; indeed, if $K_i$ leaves a proper $\mathbb{F}_p$-subspace $U_i$ of $\mathbb{F}_qe_i$ invariant, then $U_i\oplus U_i^h\oplus U_i^{h^2}$ is a proper $\mathbb{F}_p$-subspace of $V$ invariant by $H$, which contradicts the fact that $V$ is an irreducible $\mathbb{F}_pH$-module. Let $G_i=\mathbb{F}_qe_i\rtimes K_i$.

Observe that, since $h$ acts transitively on $\{\mathbb{F}_qe_1,\mathbb{F}_qe_2,\mathbb{F}_qe_3,\}$, we have $\gamma(G_1)=\gamma(G_2)=\gamma(G_3)$.

If $\gamma(G_i)\ne 2$, then Theorem~\ref{thrm:2} implies that $K_i$ is a $2$-group and $q=p^2$ for some odd prime $p$.  Thus $f=2$. As $M$ is a subdirect subgroup of $K_1\times K_2\times K_3$, $M$ is a $2$-group. Since $f=2$, $|H:M|=3$ and $H$ is nilpotent, $H=\langle h\rangle\times M$ and 
$$h=
\begin{pmatrix}
0&a&0\\
0&0&b\\
c&0&0
\end{pmatrix},$$
for some $a,b,c\in\mathbb{F}_q^\ast$. As $h^3=1$, we deduce $abc=1$. Now, if we replace $H$ with $H^g$, where 
$$g=
\begin{pmatrix}
1&0&0\\
0&a^{-1}&0\\
0&0&a^{-1}b^{-1}
\end{pmatrix},$$
we obtain $h^g=\iota_3$ and we may assume, without loss of generality, that $h=\iota_3$. Now, $$H=\langle\iota_3\rangle\times M\le {\bf C}_{\Gamma\mathrm{L}_3(q)}(\iota_3)=\langle Z,\sigma,\iota_3\rangle.$$
However, this contradicts the fact that $H$ acts irreducibly because $\mathbb{F}_q(e_1+e_2+e_3)$ is a proper subspace of $\mathbb{F}_q^3$ and is invariant by 
$\langle Z,\sigma,\iota_3\rangle$ and hence also by $H$.

Assume $\gamma(G_i)= 2$ and that  $|K_i:\langle K_i^\ast\rangle|$ is divisible by some odd prime number $r\ne 3$. 
 Let $M_i$ be the preimage of $\langle K_i^\ast\rangle$ by $\pi_i$. As $h$ normalizes $M$ and acts permuting cyclically the three direct summands $\mathbb{F}_qe_1$, $\mathbb{F}_qe_2$ and $\mathbb{F}_qe_3$ of $V$, we have $M_1^h=M_2$, $M_2^h=M_3$ and $M_3^h=M_1$.
 
Let $R$ be a Sylow $r$-subgroup of $H$ and let $Q$ be a Hall $r'$-subgroup of $H$. As $H=R\times Q$, we get $M_i=(R\cap M_i)\times (Q\cap M_i)$. Since $h$ is an element of $3$-power order and since $h$ cyclically permutes $M_1$, $M_2$ and $M_3$, we deduce
$$R\cap M_1=R\cap M_2=R\cap M_3.$$ 
Observe also that, as $r$ divides $|M:M_i|$, $R\cap M_1=R\cap M_2=R\cap M_3$ is a proper subgroup of $M$.

Since $H=R\times Q$ and since $R$ and $Q$ have relatively prime orders, we have $H^\ast\subseteq R^\ast\times Q^\ast$, see also the proof of Lemma~\ref{lemma:mercoledi33}. As $H=\langle H^\ast\rangle$, we deduce that $R^\ast$ generates $R$. However this is impossible. Indeed, since $r\ne 3$, $R\le M$ and hence $R^\ast\subseteq M^\ast\subseteq M_1\cup M_2\cup M_3$. Thus $R^\ast\subseteq R\cap (M_1\cup M_2\cup M_3)=(R\cap M_1)\cup (R\cap M_2)\cup (R\cap M_3)=R\cap M_1$. This contradiction has shown that $|K_i:\langle K_i^\ast\rangle|$ is a power of $3$. Now, $|K_i|$ from Theorem~\ref{thrm:mainmain}, we deduce that $K_i$ is a $\{2,3\}$-group. Since $M$ is a subdirect subgroup of $K_1\times K_2\times K_3$ and since $|H:M|=3$, we deduce that $H$ is a $\{2,3\}$-group.

\end{document}